\documentclass[11pt]{amsart}
\usepackage{amsmath}
\usepackage{amssymb}
\usepackage{graphicx}
\usepackage{mathrsfs}
\usepackage{tikz-cd}
\usepackage{xcolor}
\usepackage{tikz}
\usepackage[utf8]{inputenc}
\usepackage[english]{babel}
\usepackage{amsthm}

\newtheorem{thm}{Theorem}[section]
\newtheorem{prop}[thm]{Proposition}
\newtheorem{lemma}[thm]{Lemma}

\theoremstyle{definition}
\newtheorem{defn}[thm]{Definition}
\newtheorem{rmk}[thm]{Remark}

\usetikzlibrary{%
  matrix,%
  calc,%
  arrows%
}

\graphicspath{ {Images/} }

\newcommand{\C}{\mathbb{C}}

\newcommand{\Z}{\mathbb{Z}}
\newcommand{\Q}{\mathbb{Q}}

\DeclareMathOperator{\Tr}{Tr}
\DeclareMathOperator{\id}{id}

\DeclareMathOperator{\im}{im}

\DeclareMathOperator{\Ext}{Ext}

\DeclareMathOperator{\Hom}{Hom}
\DeclareMathOperator{\Gl}{GL}
\DeclareMathOperator{\Aut}{Aut}
\DeclareMathOperator{\tr}{tr}
\usepackage{esint}
\usepackage[backref=page, colorlinks]{hyperref}
\title{Frobenius Non-Stability of Nilpotent Groups}
\author{Forrest Glebe}

\begin{document}

\maketitle
\begin{abstract}
A countable discrete group is said to be Frobenius stable if every function from the group to unitary matrices that is ``almost multiplicative'' in the Frobenius norm is ``close'' to a unitary representation in the Frobenius norm. The purpose of this paper is to show that finitely generated nilpotent groups that are not virtually cyclic are not Frobenius stable. Our argument proves the same result for other unnormalized Schatten $p$-norms with $1<p\le\infty$.
\end{abstract}

\section{Introduction}

Let $\Gamma$ be a countable discrete group and let $1\le p\le\infty$. Let $||\cdot||_p$ denote the unnormalized Schatten $p$-norm, $||M||_p=(\Tr((M^*M)^{p/2}))^{1/p}$ for $p<\infty$ and operator norm for $p=\infty$. Let $U(k_n)$ be the $k_n\times k_n$ complex unitary group.  A {$p$-asymptotic homomorphism} is a sequence of functions $\rho_n:\Gamma\rightarrow U(k_n)$ so that for all $x,y\in\Gamma$ we have $||\rho_n(xy)-\rho_n(x)\rho_n(y)||_p\rightarrow0$. We say that a $p$-asymptotic representation is {\em $p$-perturbable to a genuine representation} if there is a sequence of unitary representations $\Tilde{\rho}_n:\Gamma\rightarrow U(k_n)$ so that for all $x\in\Gamma$ we have $||\Tilde{\rho}_n(x)-\rho_n(x)||_p\rightarrow0$. We say $\Gamma$ is {\em $p$-stable} if every $p$-asymptotic representation is $p$-perturbable to a genuine representation. Of particular interest is the $p=2$ case, called {\em Frobenius stability}, and the $p=\infty$ case, called {\em matricial stability}. Frobenius stability was first introduced, by Chiffre, Glebsky, Lubotzky, and Thom in~\cite{nonaproximable}. The main purpose of this paper is to prove the following result:

\begin{thm}\label{main3}
Let $\Gamma$ be a finitely generated nilpotent group and let $1<p\le\infty$. Then $\Gamma$ is $p$-stable if and only if $\Gamma$ is virtually cyclic.
\end{thm}

The proof is at the end of Section~4 of the paper. The fact that being virtually cyclic is sufficient follows from a result in~\cite{nonaproximable}, see Lemma~\ref{vfree} for an explanation.

For the torsion-free case at $p=\infty$ our result follows from a more general result of Dadarlat in~\cite{obs}. Combining this with an argument due to Bader, Lubotzky, Sauer, and Weinberger in~\cite{lattice} one can deduce our result in the $p=\infty$ case with torsion as well. See the proof of Theorem~\ref{main3} for more details.

Our result, applied to the $p=2$ case, is the last of four commonly studied ``stability properties'' to be determined for finitely generated nilpotent groups. One such notion is {\em permutation stability}. Permutation stability is defined similarly to $p$-stability, but the unitary group is replaced with $S_n$, the symmetric group on $[n]=\{1,\ldots,n\}$, and the norm metric is replaced with the normalized Hamming distance, $d_n(\sigma,\tau)=\frac1n|\{x\in[n]|\sigma(x)\ne\tau(x)\}|$. In~\cite{permutation}, Becker, Lubotzky, and Thom show that virtually polycyclic groups, and hence finitely generated nilpotent groups are permutation stable. 

Another notion of stability is {\em Hilbert-Schmidt stability} defined similarly to $p$-stability, but the norm used is the normalized Hilbert-Schmidt norm,  $||M||_{\mbox{HS}}=\sqrt{\tr(M^*M)}=||M||_2/\sqrt{k}$ where $\tr$ denotes the normalized trace and $k$ is the number of columns in $M$. It has recently been shown by Levit and Vigdorovich \cite{LVhilbertshmidt} and independently by Eckhardt and Shulman \cite{EShilbertshmidt} that finitely generated nilpotent groups groups are Hilbert-Schmidt stable. Indeed the result of Levit and Vigdorovich shows Hilbert-Shmidt stability for finitely generated virtually nilpotent groups.

To summarize it is now known whether or not finitely generated nilpotent groups are stable in all four of the most studied ways; they are all permutation stable~\cite{permutation} and Hilbert-Schmidt stable \cite{LVhilbertshmidt}\cite{EShilbertshmidt}, while virtually cyclic groups are the only such groups to be matricially stable~\cite{obs} or Frobenius stable~(Theorem~\ref{main3}).

In \cite{voiculescuM} Voiculescu showed that $\Z^2$ is not matricially stable (though the terminology had not yet been invented) by constructing an explicit sequence of pairs of unitaries that commute asymptotically in the operator norm, but remain far, in the operator norm, from pairs of unitaries that commute. By the well-known bounds $||M||\le||M||_p\le k^{1/p}||M||$, where $||\cdot||$ is the operator norm, this sequence also happens to show $\Z^2$ is not $p$-stable, for $p>1$. In \cite{stab} Eilers, Shulman, and S\o rensen, give explicit the operator norm asymptotic representations that are not perturbable, in the operator norm, to genuine representations for non-cyclic torsion-free finitely generated 2-step nilpotent groups. While they do not explicitly mention the Frobenius norm or other $p$-norms, their example can easily be seen to be a $p$-asymptotic representation, for $1<p\le\infty$, and thus proves that these groups are not $p$-stable for the same reason.

As mentioned, it is proved in \cite{obs} that torsion-free finitely generated nilpotent groups, other than $\Z$ and the trivial group are not stable in operator norm. In~\cite{constructive} an explicit formula, in terms of cohomological data, of an operator norm asymptotic representation, that cannot be perturbed, in the operator norm, to a genuine representation, is given for a class of groups that contains all non-cyclic torsion-free finitely generated nilpotent groups. Unfortunately, this formula is not asymptotically multiplicative in the Frobenius norm. In this paper we modify the construction to get asymptotic multiplicativity in Frobenius norm and other Schatten $p$-norms.

To be more explicit, for our formula and the one in \cite{constructive} the operator norm defect to being multiplicative, $||\rho_n(xy)-\rho_n(x)\rho_n(y)||$, is bounded by $2\pi|\sigma(x,y)|/n$ where $\sigma$ is a 2-cocycle on the group (see Proposition~\ref{asymp}). In~\cite{constructive} this acts on an $n^m$ dimensional vector space, where $m$ is the Hirsch length of the group, so this does not imply asymptotic multiplicativity in the Frobenius norm. In this paper the dimension of the space that $\rho_n$ acts on is reduced to $n$ which implies that the defect to being multiplicative in the Frobenius norm is bounded above by~$2\pi|\sigma(x,y)|/\sqrt{n}$. In general the defect to being multiplicative in an unnormalized Schatten $p$-norm, for $p<\infty$, is $2\pi|\sigma(x,y)|n^{\frac1p-1}$, which goes to zero for $p>1$.

The paper is organized as follows. Section 2 has relevant background information. Section 3 introduces a condition on a cohomology class (Definition~\ref{skinny}) and proves that all non-cyclic torsion-free finitely generated nilpotent groups have a cohomology class of a this form that is non-torsion (Theorem~\ref{main1}). In Section 4 we explain why this implies the main result by modifying arguments from \cite{constructive}. First we prove the main result in the torsion-free case by making a formula for a $p$-asymptotic representation (Proposition~\ref{formula}) in terms of a cocycle in the relevant cohomology class then showing it cannot be perturbed to a genuine representation (Theorem~\ref{main2}). Finally, at the end of Section~4 we give a proof of Theorem~\ref{main3}, by reducing to the torsion-free case.

\section{Background Information}

\subsection{Notation}
We will use $e$ as the identity of a multiplicative group, and $0$ for the identity of $\Z$. This is to avoid confusion with $1$ denoting the generator of the group $\Z$. We will also use $e$ for the trivial group in commutative diagrams. 

Throughout the paper $||\cdot||_p$ will denote the unnormalized Schatten $p$-norm, and $||\cdot||$ will denote the operator norm, but in some circumstances we will write $||\cdot||_p$ and allow for the case $p=\infty$ to consider operator norm simultaneously with other norms.

We will use $U(n)$ to denote the $n\times n$ complex unitary group.

When we have a nested (semi)-direct product we will use coordinates to describe elements. For example if $\Gamma$ is an arbitrary discrete group then a typical element of $(\Z\times\Gamma)\rtimes_\gamma\Z$ is written as $(x,y,z)$ with $x,z\in\Z$ and $y\in\Gamma$. The identity of this group could be written as $(0,e,0)$. When writing a semi-direct product $\Gamma\rtimes_\gamma\Z$ we will simply specify an element $\gamma\in\Aut(\Gamma)$ so that the action of $\Z$ on $\Gamma$ is determined by $1\mapsto\gamma$. 

\subsection{Group (Co)Homology}
There are many ways to characterize group homology and cohomology, but to us the most useful will be to describe them as the homology and cohomology of an explicit chain complex described below. We will only use homology and cohomology with coefficients in $\Z$ and the trivial action in this paper. For more about this construction see \cite[Chapter II.3]{coho}.

\begin{defn}
Let $\Gamma$ be a discrete group. For $n>0$ we define $C_n(\Gamma)$ to be the free abelian group generated by elements of $\Gamma^n$, and define $C_0(\Gamma)=\Z$. We may write an element of $\Gamma^n$ as $[a_1|a_2|\cdots|a_n]$ with $a_i\in\Gamma$. We thus write a typical element of $C_n(\Gamma)$ as
$$c=\sum_{i=1}^Nx_i[a_{i1}|a_{i2}|\cdots|a_{in}]$$
with $x_i\in\Z$ and $a_{ij}\in\Gamma$. We define the boundary map $\partial_n$ from $C_n(\Gamma)\rightarrow C_{n-1}(\Gamma)$ by
$$\partial_n[a_1|\cdots|a_n]=$$
\vspace{-7mm}
$$[a_2|\cdots|a_{n}]+\sum_{i=1}^{n-1}(-1)^i[a_1|\cdots|a_{i-1}|a_ia_{i+1}|a_{i+2}|\cdots|a_n]+(-1)^n[a_1|\cdots|a_{n-1}]$$
for $n>1$ and $\partial_1=0$. Often we will just write $\partial$ where the domain is clear from context. The group {\em homology} of $\Gamma$ is the homology group the chain complex $(C_\bullet(\Gamma),\partial_\bullet)$. \end{defn}

For readers who are less familiar with homological algebra we will be more explicit. It happens that $\im(\partial_n)\subseteq\ker(\partial_{n+1})$ so we define the $n$th homology group of $\Gamma$ to be
$$H_n(\Gamma):=\ker(\partial_n)/\im(\partial_{n-1}).$$
An element of $\ker\partial_n$ is called an {\em $n$-cycle}.

\begin{defn}
The {\em cohomology of $\Gamma$} is the cohomology of $(C_\bullet,\partial_\bullet)$ with coefficients in $\Z$.
\end{defn}

We will again be more explicit for readers who are less familiar with homological algebra. We use the notation
$$C^n(\Gamma):=\Hom(C_n(\Gamma),\Z)$$
and note that this is isomorphic to the group of functions from $\Gamma^n$ to~$\Z$ for $n>0$ and isomorphic to $\Z$ for $n=0$. Then we define $\partial^n:C^n(\Gamma)\rightarrow C^{n+1}(\Gamma)$ to be the adjoint\footnote{Actually \cite{coho} defines the coboundary map to be $(-1)^{n+1}$ times the adjoint, but this does not change the image or kernel boundary of the maps so it leads to an equivalent definition of the cohomology groups.} of $\partial_{n+1}$. This gives us
$$(\partial^n\sigma)(a_1,\ldots,a_{n+1})=$$
\vspace{-7mm}
$$\sigma(a_2,\ldots, a_{n+1})+\sum_{i=1}^n(-1)^i\sigma(a_1,\ldots, a_{i-1},a_ia_{i+1},a_{i+2},\ldots, a_{n+1})+(-1)^{n+1}\sigma(a_1,\ldots,a_n)$$
for $n>0$ and $\partial^0=0$. Then we define 
$$H^n(\Gamma):=\ker(\partial^n)/\im(\partial^{n+1}).$$
An element of $\ker(\partial^n)$ is called an {\em $n$-cocycle}. If $\sigma$ is an $n$-cocycle then we denote its cohomology class as $[\sigma]$.
\begin{defn}
We say an $n$-cocycle $\sigma$ is {\em normalized} if for each $i$ we have that $a_i=e$ implies that $\sigma(a_1,\ldots ,a_n)=0$.
\end{defn}
As with homology we will suppress the $n$ in $\partial^n$ if the domain is obvious from context.

Suppose that $f$ is a group homomorphism from $\Gamma_1$ to $\Gamma_2$. This induces  a map $f_*:H_n(\Gamma_1)\rightarrow H_n(\Gamma_2)$ defined by the formula
$$\sum_{i=1}^Nx_i[a_{i1}|\cdots|a_{in}]\mapsto\sum_{i=1}^Nx_i[f(a_{i1})|\cdots|f(a_{in})].$$
Similarly there is a map from $f^*:H^n(\Gamma_2)\rightarrow H^n(\Gamma_1)$ defined by
$$(f^*(\sigma))(a_1,\ldots a_n)=\sigma(f(a_1),\ldots, f(a_n)).$$
Both maps are functorial.

In general cohomology is a bifunctor by allowing ``coefficients'' in an arbitrary abelian group with a $\Gamma$-action. For us that will mean that if $f:\Z\rightarrow\Z$ is a group homomorphism this induces a map $f_*:H^n(\Gamma)\rightarrow H^n(\Gamma)$. It must be the case that $f$ is multiplication by some integer $k$, and it happens that~$f_*$ is also multiplication by $k$. 

Because $C^n(\Gamma)$ is isomorphic to $\Hom(C_n(\Gamma),\Z)$ there is a natural bilinear map from $C^n(\Gamma)\times C_n(\Gamma)\rightarrow \Z$ defined by
$$\left\langle\sigma,\sum_{i=1}^Nx_i[a_{i1}|\cdots|a_{in}]\right\rangle=\sum_{i=1}^Nx_i\sigma(a_{i1},\ldots,a_{in}).$$
This descends to a well-defined bilinear map from $H^n(\Gamma)\times H_n(\Gamma)\rightarrow \Z$. We will use the notation $\langle\cdot,\cdot\rangle$ for both maps. The maps $f_*$ and $f^*$ are adjoint in the following sense
$$\langle f^*(\sigma),c\rangle=\langle\sigma,f_*(c)\rangle.$$

\begin{defn}
If $\Gamma$ is a discrete group a {\em central extension} of $\Gamma$ by $\Z$ is a short exact sequence
$$\begin{tikzcd}
e\arrow{r} & \Z
\arrow{r}{} &
\Tilde{\Gamma} \arrow{r}{}
& \Gamma\arrow{r}{}& e
\end{tikzcd}$$
where the image of $\Z$ in $\tilde{\Gamma}$ is central in $\tilde{\Gamma}$. We say two central extensions are {\em equivalent} if we can make a commutative diagram as follows:
$$\begin{tikzcd}
e\arrow{r} & \Z
\arrow{r}{}\arrow{d}{\id_\Z} &
\Tilde{\Gamma}_1\arrow{d}{\cong} \arrow{r}{}
& \Gamma\arrow{r}{}\arrow{d}{\id_\Gamma}& e
\\
e\arrow{r} & \Z
\arrow{r}{} &
\Tilde{\Gamma}_2 \arrow{r}{}
& \Gamma\arrow{r}{}& e.
\end{tikzcd}$$
\end{defn}

\begin{thm}[{\cite[Theorem IV.3.12]{coho}}] As a set $H^2(\Gamma)$ is in bijection with the equivalence classes of central extension of $\Gamma$ by $\Z$.
\end{thm}

The zero element of the cohomology group corresponds to the direct product extension:
$$\begin{tikzcd}
e\arrow{r} & \Z
\arrow{r}{} &
\Z\times\Gamma \arrow{r}{}
& \Gamma\arrow{r}{}& e
\end{tikzcd}$$
with the obvious maps.

Given an explicit central extension we may find a cocycle representative of the corresponding element of $H^2(\Gamma)$ as follows. Pick $\theta$ to be a set theoretic section from $\Gamma$ to $\Tilde{\Gamma}$. Then viewing $\Z$ as a subset of $\Tilde{\Gamma}$ define
$$\sigma(g,h)=\theta(g)\theta(h)\theta(gh)^{-1}\in \Z.$$
By \cite[Equation IV.3.3]{coho} this is a cocycle representative of the cohomology class corresponding to this central extension.

If $[\sigma]\in H^2(\Gamma)$ then the central extension of corresponding to $[\sigma]$ will be denoted
$$\begin{tikzcd}
e\arrow{r}{}&\Z\arrow{r}{\iota_\sigma}&\Gamma_\sigma\arrow{r}{\varphi_\sigma}&\Gamma\arrow{r}{}&e
\end{tikzcd}.$$
When we use the symbols $\iota_\sigma$, $\Gamma_\sigma$, or $\varphi_\sigma$ we are referring to the relevant parts of this diagram.

We will use the following result

\begin{prop}\label{phiker}
The kernel of $\varphi_\sigma^*:H^2(\Gamma)\rightarrow H^2(\Gamma_\sigma)$ is the generated by~$[\sigma]$. 
\end{prop}
\begin{proof}
The Gysin sequence is a long exact sequence coming from the Hochschild-Serre spectral sequence with the following form
$$\begin{tikzcd}
\cdots H^n(\Gamma)\arrow{r}{\smile[\sigma]}&H^{n+2}(\Gamma)\arrow{r}{\varphi_\sigma^*}&H^{n+2}(\Gamma_\sigma)\arrow{r}{}&H^{n+1}(\Gamma)\arrow{r}{\smile[\sigma]}&H^{n+3}(\Gamma)\cdots
\end{tikzcd}$$
where $\smile$ is the cup product \cite[Theorem 3, Theorem 4]{gysin}. See \cite[Chapter V.3]{coho} for a full description of the cup product. From exactness of this sequence looking at $n=0$ we get that the kernel of $\varphi_\sigma^*$ is $H^0(\Gamma)\smile[\sigma]$. Since $H^0(\Gamma)$ is generated by the class of the constant 1 cocycle, which is the identity element for the cup product ring, we have that $H^0(\Gamma)\smile[\sigma]$ is precisely the span of $[\sigma]$. 
\end{proof}

\subsection{Torsion-Free Finitely Generated Nilpotent Groups}
Let $\Gamma$ be a torsion-free finitely generated nilpotent group.
\begin{defn}
A {\em Mal'cev basis} for $\Gamma$ is an $m$-tuple of elements, $(a_1,\ldots ,a_m)\in\Gamma^m$ which satisfies the following conditions
\begin{itemize}
    \item For all $g\in\Gamma$, $g$ can be written uniquely as $g=a_1^{x_1}\cdots a_m^{x_m}$ for some $(x_1,\ldots, x_m)\in\Z^m$.  We call this presentation the {\em canonical form} of~$g$.
    \item The subgroups $\Gamma_i=\langle a_i,\ldots, a_m\rangle$ form a central series for $\Gamma$.
\end{itemize}
\end{defn}

Every finitely generated torsion-free nilpotent group has a Mal'cev basis by \cite[Lemma 8.23]{Comp}. When working with a torsion-free finitely generated nilpotent group with a specified Mal'cev basis we will treat elements of the group as $m$-tuples of integers $x=(x_1,\ldots, x_m)$ where the $x_i$ denote the exponents in the canonical form for $x$. 

\begin{defn}\label{hirsch1}
The cardinality of the Mal'cev basis is called the {\em Hirsch length} of the group.
\end{defn}

This is independent of the choice of Mal'cev basis \cite[Exercise 8]{polycyclic}.

In order to extend this definition to finitely generated nilpotent groups with torsion we use the following well-known result.

\begin{prop}\label{torsion}\cite[Lemma~8.2.2]{wordprocessing}
If $\Gamma$ is a nilpotent group, and $T$ is the set of torsion elements, then $T$ is a finite normal subgroup of $\Gamma$, and $\Gamma/T$ is torsion-free.
\end{prop}

\begin{defn}\label{hirsch2}
If $\Gamma$ is a finitely generated nilpotent group, and $T$ is its torsion-subgroup then we define the {\em Hirsch Length} of $\Gamma$ to be the Hirsch length of~$\Gamma/T$, defined in Definition~\ref{hirsch1}.
\end{defn}

This coincides with the definition of Hirsch length defined more generally in~\cite{elementary}. 

\begin{prop}\label{vcyc}
Let $\Gamma$ be a finitely generated nilpotent group with Hirsch length at most 1. Then $\Gamma$ is virtually cyclic.
\end{prop}

\begin{proof}
There must be a torsion-free subgroup $N\le\Gamma$ that is finite-index and torsion-free by \cite[Theorem~2.1]{notesNG}. Since the Hirsch length $\Gamma$ is at most 1, the Hirsch length of $N$ is also at most 1 by \cite[Theorem~1b]{elementary} so either $N\cong\Z$ or $N=\{e\}$, which implies $\Gamma$ is virtually cyclic. 
\end{proof}

We will use the following result. Recall that an element of a torsion-free finitely generated nilpotent group may be represented as an $m$-tuple of integers.

\begin{thm}\cite[Theorem 6.5]{hall}\label{polynomial}
If $\Gamma$ is a torsion-free finitely generated nilpotent group with Hirsch length $m$ and coordinates given by a Mal'cev basis then there are polynomials $q_1,\ldots, q_m\in\Q[x_1,\ldots, x_m,y_1,\ldots, y_m]$ so that
$$xy=(q_1(x,y),\ldots,q_m(x,y)).$$
\end{thm}

\begin{prop}\label{liftbasis}\cite[Proposition 4.1]{constructive}
If $\Gamma$ is a torsion-free finitely generated nilpotent group with Mal'cev basis $a_1,\ldots,a_m$ and $[\sigma]\in H^2(\Gamma)$ then a Mal'cev basis, $b_1,\ldots, b_{m+1}$ for $\Gamma_\sigma$ can be found as follows. Let $b_i$ be any lift of $a_i$ for $1\le i\le m$ and let $b_{m+1}=\iota_\sigma(1)$.  
\end{prop}

\begin{defn}
Let $\Gamma$ be a finitely generated, torsion-free, nilpotent group. If $a_1,\ldots, a_m$ is a Mal'cev basis for $\Gamma$ then $\alpha:\Gamma\rightarrow\Z$ defined by $\alpha(a_1^{x_1}\cdots a_m^{x_m})=x_1$ is called the {\em canonical homomorphism with respect to $a_1,\ldots, a_m$}.
\end{defn}

\begin{prop}\label{quotientB}
Let $\Gamma$ be a torsion-free finitely generated nilpotent group with Mal'cev basis $a_1,\ldots, a_m$ and let $q:\Gamma\rightarrow\Gamma/\langle a_m\rangle$ be the quotient map. Then $q(a_1),\ldots, q(a_{m-1})$ is a Mal'cev basis for $\Gamma/\langle a_m\rangle$. If $\alpha:\Gamma/\langle a_m\rangle\rightarrow\Z$ is the canonical homomorphism with respect to $q(a_1),\ldots,q(a_{m-1})$ then $\alpha\circ q:\Gamma\rightarrow\Z$ is the canonical homomorphism with respect to $a_1,\ldots, a_m$.
\end{prop}
\begin{proof}
Call $\bar{g}=q(g)$. If $\bar{g}\in\Gamma/\langle a_m\rangle$ we have $g=a_1^{x_1}\cdots a_m^{x_m}$ so $\bar{g}=\bar{a}_1^{x_1}\cdots\bar{a}_{m-1}^{x_{m-1}}$. To show that this is unique suppose that
$$\bar{a}_1^{x_1}\cdots\bar{a}_{m-1}^{x_{m-1}}=\bar{a}_1^{y_1}\cdots\bar{a}_{m-1}^{y_{m-1}}$$
thus
$${a}_1^{x_1}\cdots{a}_{m-1}^{x_{m-1}}={a}_1^{y_1}\cdots{a}_{m-1}^{y_{m-1}}\cdot a_m^{k}$$
so it follows that for $x_i=y_i$ for each $i$. Let $\Gamma_i=\langle a_i,\ldots, a_m\rangle$ and let $\bar{\Gamma}_i=\langle\bar{a}_i,\ldots, \bar a_{m-1}\rangle=\Gamma_i/\langle a_m\rangle$. Note that 
$$[\bar{\Gamma},\bar{\Gamma}_i]=[\Gamma,\Gamma_i]/\langle a_m\rangle\subseteq\Gamma_{i+1}/\langle a_m\rangle=\bar{\Gamma}_{i+1}.
$$
To show the last claim
$$\alpha(q(a_1^{x_1}\cdots a_m^{x_m}))=\alpha(\bar{a}_1^{x_1}\cdots\bar{a}_{m-1}^{x_{m-1}})=x_1.$$
\end{proof}
\subsection{The Winding Number Argument}

The classical example due to Voiculescu in \cite{voiculescuM} for an asymptotic representation (in either operator or Frobenius norm) of $\Z^2$ that is not perturbable (in either norm) to a sequence of genuine representations comes in the form:
$$\rho_n(a,b)=u_n^av_n^b$$
where $u_n$ and $v_n$ are $n\times n$ matrices such that
$$u_n=\begin{bmatrix}
0&0&\cdots&0&0&1\\
1&0&\cdots&0&0&0\\
0&1&\cdots&0&0&0\\
\vdots&\vdots&\ddots&\vdots&\vdots&\vdots\\
0&0&\cdots&1&0&0\\
0&0&\cdots&0&1&0
\end{bmatrix}\mbox{ and }
v_n=\begin{bmatrix}
\exp\left(\frac{2\pi i}{n}\right)&0&0&\cdots&0\\
0&\exp\left(\frac{4\pi i}{n}\right)&0&\cdots&0\\
0&0&\exp\left(\frac{6\pi i}{n}\right)&\cdots&0\\
\vdots&\vdots&\vdots&\ddots&\vdots\\
0&0&0&\cdots&1
\end{bmatrix}.
$$
It can be computed that $u_nv_nu_n^{-1}v_n^{-1}=\exp\left(\frac{-2\pi i}{n}\right)\id_n$. Thus we see that for $1\le p<\infty$
\begin{align*}
||\rho_n(x)\rho_n(y)-\rho_n(x+y)||_p&=||\rho_n(x)\rho_n(y)\rho_n(x+y)^{-1}-\id_{\C^n}||_p\\
&=||u_n^{x_1}v_n^{x_2}u_n^{y_1}v_n^{y_2}v_n^{-x_2-y_2}u_n^{-x_1-y_1}-\id_{\C^n}||_p\\
&=||u_n^{x_1}(v_n^{x_2}u_n^{y_1}v_n^{-x_2}u_n^{-y_1})u_n^{-x_1}-\id_{\C^n}||_p\\
&=\left|\left|\left(\exp\left(\frac{2\pi i x_2y_1}{n}\right)-1\right)\id_{\C^n}\right|\right|_p\\
&\le2\pi |x_2y_1|n^{\frac1p-1}.
\end{align*}
Voiculescu shows that $u_n$ and $v_n$ cannot be arbitrarily close to a pair of unitaries that genuinely commute with each other, in the operator norm. Because the operator norm is smaller than any other unnormalized Schatten $p$-norm they cannot be close to commuting matrices in these norms either.

The argument we will use to show non-perturbability in this paper has its roots in the ``winding number argument'' first discovered by Kazhdan in~\cite{KvoiculescuM} and independently used by Exel and Loring in \cite{ELVoiculescuM} that $u_n$ and $v_n$ are far from any pair of matrices that commute. A more general statement of this argument can be found in \cite[Theorem 3.9]{stab} or in~\cite{quasireps}. We will use a more homological version of the result described below.
\begin{defn}\label{pairing}
If
$$c=\sum_{j=1}^N x_j[a_j|b_j]\in C_2(\Gamma)$$
and $\rho:\Gamma\rightarrow\Gl_n(\C)$ so that for all $j\in\{1,\ldots,N\}$
$$||\rho(a_jb_j)\rho(a_j)^{-1}\rho(b_j)^{-1}-\id_{\C^n}||<1$$
we define
$$\langle \rho,c\rangle=\frac{1}{2\pi i}\sum_{j=1}^Nx_j\Tr(\log(\rho(a_jb_j)\rho(b_j)^{-1}\rho(a_j)^{-1}))$$
where $\log$ is defined as a power series centered at 1.
\end{defn}

If $\partial c=0$ we have that $\langle\rho,c\rangle\in\Z$  \cite[Proposition 3.4]{constructive}. The version of the winding number argument we are using is as follows.

\begin{thm}\label{obstruction}\cite[Theorem 3.7]{constructive}
If $\rho_0$ is a (not necessarily unitary) representation of $\Gamma$, $\rho_1$ is a function from~$\Gamma$ to~$U(n)$, 
$$c=\sum_{j=1}^N x_j[a_j|b_j]$$
is a 2-cycle on $\Gamma$, and $$||\rho_1(g)-\rho_0(g)||<\frac{1}{24}$$
for all $g\in\{a_j,b_j,a_jb_j\}_{j=1}^N$ then $\rho_1$ is multiplicative enough for $\langle\rho_1,c\rangle$ to be defined and $\langle\rho_1,c\rangle=0$.
\end{thm}

\subsection{$p$-Stability Results}

The following result is essentially \cite[Lemma~2.13]{lattice}. They state the result only for Frobenius norm, but their argument holds for other unnormalized Schatten $p$-norms, including operator norm. We include a proof for the sake of completeness and to highlight the construction of our asymptotic representation.

\begin{lemma}\label{quotient}
Suppose that $\Gamma$ is a countable discrete group and $N$ is a finite normal subgroup. Let $q:\Gamma\rightarrow\Gamma/N$ be the quotient map and let $1\le p\le\infty$. If~$\varphi_n$ is a $p$-asymptotic representation of $\Gamma/N$ then $\varphi_n\circ q$ is a $p$-asymptotic representation of $\Gamma$; if $\varphi_n\circ q$ is $p$-perturbable to a genuine representation so is $\varphi_n$. In particular, if $\Gamma$ is $p$-stable then $\Gamma/N$ is $p$-stable.
\end{lemma}

\begin{proof}
Clearly $\varphi_n\circ q$ is a $p$-asymptotic homomorphism on~$\Gamma$. Suppose now that $\varphi\circ q$ is $p$-perturbable to a genuine representation. Without loss of generality we may assume that $\varphi_n(e)=1$. By assumption, there is a sequence of group homomorphisms $\psi_n:\Gamma\rightarrow U(k_n)$ so that for all $x\in\Gamma$ we have $||\psi_n(x)-\varphi_n(q(x))||_p\rightarrow0$. Now pick $\varepsilon>0$ so that for each nontrivial irreducible representation $\alpha$ of~$N$ we have that $\max_{x\in N}||\alpha(x)-1||_p>\varepsilon$. Then for large enough $n$, we have that for all $x\in N$
$$||\psi_n(x)-1||_p=||\psi_n(x)-\varphi_n(q(x))||_p<\varepsilon.$$
If $\alpha$ is any irreducible subrepresentation of $\psi_n|_N$, for all $x\in N$ we have
$$||\alpha(x)-1||_p\le||\psi_n(x)-1||_p<\varepsilon.$$
It follows that every irreducible subrepresentation of $\psi_n|_N$ is trivial so $N\le\ker(\psi_n)$. Thus $\psi_n$ factors through $\Gamma/N$ and $\varphi_n$ is $p$-perturbable to a genuine representation.
\end{proof}

\begin{lemma}\label{vfree}
Let $\Gamma$ be a finitely generated virtually free group and $1\le p\le\infty$. Then $\Gamma$ is $p$-stable.
\end{lemma}
\begin{proof}
By~\cite[Remark 5.2]{nonaproximable} if the 2-cohomology of $\Gamma$ with coefficients in a class of Banach spaces with a $\Gamma$ action vanishes then $\Gamma$ is $p$-stable. Since the rational cohomological dimension of virtually free groups is 1 it follows from this that they must be $p$-stable.
\end{proof}

\section{Skinny Cohomology Classes}
\begin{defn}\label{skinny}
Let $[\sigma]\in H^2(\Gamma)$ and let $\alpha\in\Hom(\Gamma,\Z)$. Let $\kappa_\alpha$ be the inclusion map from $\ker(\alpha)$ to $\Gamma$. We call $[\sigma]$ {\em skinny with respect to~$\alpha$} if $\kappa_\alpha^*([\sigma])=0$.
\end{defn}

\begin{prop}\label{skinnys}
Let $\Gamma$ be a discrete group, $\alpha\in\Hom(\Gamma,\Z)$, and $[\sigma]\in H^2(\Gamma)$. The following are equivalent:
\begin{enumerate}
    \item $[\sigma]$ has a normalized cocycle representative $\sigma$ with the property that $\sigma(x,y)$ depends only on $x$ and $\alpha(y)$.
    \item $[\sigma]$ is skinny with respect to $\alpha$.
    \item There is a commutative diagram as follows
    $$\begin{tikzcd}
e\arrow{r} & \Z \arrow{r}{} 
\arrow{d}{\id_\Z} &\Z\times\ker(\alpha)\arrow{r}{}
\arrow{d}{\psi} 
&\ker(\alpha) \arrow{r}{}\arrow{d}{\kappa_\alpha}& e
\\
e\arrow{r} & \Z\arrow{r}{\iota_\sigma}
&
\Gamma_\sigma\arrow{r}{\varphi_\sigma}
&\Gamma \arrow{r}{}& e
\end{tikzcd}$$
with the obvious maps on the top row.
\end{enumerate}
If $\Gamma$ is a finitely-generated, torsion-free, nilpotent group and $\alpha$ is the canonical homomorphism with respect to a Mal'cev basis then the representative in (1) can be taken to be polynomial with respect to this basis. That is to say there is $q\in\Q[x_1,\ldots,x_m,y_1]$ so that $\sigma(x,y)=q(x_1,\ldots,x_m,y_1)$ for all $x,y\in\Gamma$.
\end{prop}
Note that the cocycle will always map integer inputs to integer outputs, but in general the coefficients may be non-integral rationals.
\begin{proof}
(1)$\Longrightarrow$(2) Let $\sigma$ be a representative as in (1). To show (2) it suffices to show that for $x,y\in\ker(\alpha)$ we have that $\sigma(x,y)=0$. Note that since $y\in\ker(\alpha)$ and $\sigma$ depends only on $x$ and $\alpha(y)$ we have $\sigma(x,y)=\sigma(x,e)=0$ since $\sigma$ is normalized.

\noindent
(2)$\Longrightarrow$(3) The central extension corresponding to $\kappa_\alpha^*([\sigma])$ must fit into a commutative diagram as in the top row by \cite[Chapter IV.3 Exercise 1]{coho}. Since $\kappa_\alpha^*([\sigma])$ is zero it must correspond to the trivial central extension, so the commutative diagram as written must exist.

\noindent (3)$\Longrightarrow$(1) Let $a_1\in\alpha^{-1}(1)$. In the torsion-free, finitely generated, nilpotent case we can pick $a_1$ to be the first element of the Mal'cev basis. Any element in $\Gamma$ may uniquely be written as $a_1^{x_1}y$ with $y\in\ker(\alpha)$. Let $b_1\in\Gamma_\sigma$ be a lift of $a_1$. Define a set-theoretic section $\theta$ as follows:
$$a_1^{x_1}y\mapsto b_1^{x_1}\psi(0,y).$$
Commutativity of the right square in the commutative diagram shows that $\varphi_\sigma(\psi(0,y))=y$. Combining this and the fact that $\varphi_\sigma(b_1)=a_1$ we get that~$\theta$ is in fact a set theoretic section of $\varphi_\sigma$. Then we define $\sigma'$ by the formula $\theta(g)\theta(h)=\theta(gh)\iota_\sigma(\sigma'(g,h))$. This is a cocycle representative of $[\sigma]$ by \cite[Equation IV.3.3]{coho}. Let $g=a_1^{x_1}y$ and $h=a_1^{x_2}z$ with $y,z\in\ker(\alpha)$. Now we compute that
\begin{align*}
\theta(a_1^{x_1} y)\theta(a_1^{x_2} z)&\theta(a_1^{x_1}ya_1^{x_2}z)^{-1}\\
&=b_1^{x_1}\psi(0,y)b_1^{x_2}\psi(0,z)\theta(a_1^{x_1+x_2}(a_1^{-x_2}ya_1^{x_2})z)^{-1}\\
&=b_1^{x_1}\psi(0,y)b_1^{x_2}\psi(0,z)\psi(0,z)^{-1}\psi(0,a_1^{-x_2}ya_1^{x_2})^{-1}b_1^{-x_1-x_2}\\
&=b_1^{x_1}\psi(0,y)b_1^{x_2}\psi(0,a_1^{-x_2}ya_1^{x_2})^{-1}b_1^{-x_1-x_2}.
\end{align*}
The last expression clearly does not depend on $z$. Thus $\sigma'(g,h)$ only depends on $g$ and $\alpha(h)=x_2$. Note that $\sigma'$ is normalized because $\theta(e)=e$ \cite[Chapter IV equation 3.4]{coho}. 

All that remains to be shown is that $\sigma'$ is polynomial with respect to the Mal'cev basis in the torsion-free finitely-generated nilpotent case. To show this we let $a_1,\ldots,a_m$ be a Mal'cev basis for $\Gamma$ and compute that
\begin{align*}
\theta(a_1^{x_1}\cdot a_2^{x_2}\cdots a_m^{x_m})&=b_1^{x_1}\psi(0,a_2^{x_2}\cdots a_m^{x_m})\\
&=\theta(a_1)^{x_1}\theta(a_2)^{x_2}\cdots\theta(a_m)^{x_m}.
\end{align*}
Note that $\theta(a_1),\ldots,\theta(a_m),\iota_\sigma(1)$ makes a Mal'cev basis for $\Gamma_\sigma$ by Proposition~\ref{liftbasis}. Using the Mal'cev bases mentioned above for both $\Gamma$ and $\Gamma_\sigma$ we may express elements of both groups as $m$-tuples and $m+1$-tuples of integers. In this notation $\theta(x_1,\ldots,x_m)=(x_1,\ldots,x_m,0)$. By Theorem~\ref{polynomial} multiplication in $\Gamma_\sigma$ is given by $x\cdot y=(q_1(x,y),\ldots,q_{m+1}(x,y))$ with $q_i\in\Q[x_1,\ldots,x_{m+1},y_1,\ldots,y_{m+1}]$. Then
\begin{align*}
\theta(xy)\iota_\sigma(\sigma'(x,y))&=\theta(x)\theta(y)\\
&=(q_1(\theta(x),\theta(y)),\ldots ,q_{m+1}(\theta(x),\theta(y))).
\end{align*}
The equality in the last coordinate gives us that 
$$\sigma'(x,y)=q_{m+1}(x_1,\ldots,x_m,0,y_1,\ldots,y_m,0).$$
\end{proof}

\begin{lemma}\label{pullback}
Let $\varphi:\Gamma\rightarrow\Lambda$ be a group homomorphism. If $[\sigma]\in H^2(\Lambda)$ is skinny with respect to $\alpha\in\Hom(\Lambda,\Z)$ then $\varphi^*([\sigma])$ is skinny with respect to $\alpha\circ\varphi$.
\end{lemma}
\begin{proof}
Note that $\varphi\circ\kappa_{\alpha\circ\varphi}=\kappa_{\alpha}\circ\varphi|_{\ker(\alpha\circ\varphi)}$. Thus
$$\kappa_{\alpha\circ\varphi}^*(\varphi^*([\sigma]))=\varphi|_{\ker(\alpha\circ\varphi)}^*(\kappa_\alpha^*([\sigma]))=0.$$
\end{proof}

\begin{lemma}\label{hard}
Suppose that $\Gamma$ is a discrete group and $[\sigma]\in H^2(\Gamma)$ is skinny with respect to $\alpha$, a surjective element of $\Hom(\Gamma,\Z)$. There is a non-torsion cohomology class $[\omega]\in H^2(\Gamma_\sigma)$ so that $[\omega]$ is skinny with respect to $\alpha\circ\varphi_\sigma$. Moreover, if $a$ is any element of $\Gamma_\sigma$ so that $\alpha(\varphi_\sigma(a))=1$ and $c_k$ is the 2-cycle on $\Gamma_\sigma$ defined by
$$[a|\iota_\sigma(k)]-[\iota_\sigma(k)|a]$$
then $\langle\omega,c_k\rangle=k$.
\end{lemma}
\begin{proof}
By condition (3) of Proposition~\ref{skinnys} we have a commutative diagram as follows:
$$\begin{tikzcd}
e\arrow{r} & \Z \arrow{r}{} 
\arrow{d}{\id_\Z} &\Z\times\ker(\alpha)\arrow{r}{}
\arrow{d}{\psi} 
&\ker(\alpha) \arrow{r}{}\arrow{d}{\kappa_\alpha}& e
\\
e\arrow{r} & \Z\arrow{r}{\iota_\sigma}
&
\Gamma_\sigma\arrow{r}{\varphi_\sigma}
&\Gamma \arrow{r}{}& e.
\end{tikzcd}$$
By the five lemma $\psi$ is an injection. First note that
$$\alpha(\varphi_\sigma(\psi(x,y)))=\alpha(\kappa_\alpha(y))=0.$$
Next suppose that $\alpha(\varphi_\sigma(x))=0$. Then $\varphi_\sigma(x)\in\ker(\alpha)$ so let $y=\psi(0,\varphi_\sigma(x))$. Since $\varphi_\sigma(y)=\varphi_\sigma(x)$ we have that $xy^{-1}=\iota_\sigma(n)$. Then we see that
$$\psi(n,\varphi_\sigma(x))=\psi(n,0)\psi(0,\varphi_\sigma(x))=\iota_\sigma(n)y=x.$$
Thus we have a short exact sequence
$$\begin{tikzcd}
e\arrow{r}{}&\Z\times\ker(\alpha)\arrow{r}{\psi}&\Gamma_\sigma\arrow{r}{\alpha\circ\varphi_\sigma}&\Z\arrow{r}{}&e
\end{tikzcd}.$$
Because $\Z$ is a free group, this sequence must split. Hence we have
$$\Gamma_\sigma\cong(\Z\times\ker(\alpha))\rtimes_\gamma\Z$$
where $\gamma$ is some element of $\Aut(\Z\times\ker(\alpha))$. One explicit isomorphism is given by $(x,y,z)\mapsto\psi(x,y)a^z$. 
Define $\eta\in\Aut(\Z\times\Z\times\ker(\alpha))$ by the formula
$$\eta(x,y,z)=(x+y,\gamma(y,z)).$$
To show that this is a homomorphism we compute 
\begin{align*}
\eta(x_1,y_1,z_1)\cdot\eta(x_2,y_2,z_2)&=(x_1+y_1,\gamma(y_1,z_1))\cdot(x_2+y_2,\gamma(y_2,z_2))\\
&=(x_1+x_2+y_1+y_2,\gamma(y_1+y_2,z_1z_2))\\
&=\eta(x_1+x_2,y_1+y_2,z_1z_2)\\
&=\eta((x_1,y_1,z_1)\cdot(x_2,y_2,z_2)).
\end{align*}
Moreover the inverse of $\eta$ is given by
$$(x,y,z)\mapsto(x-w,\gamma^{-1}(y,z))$$
where $w$ is the first component of $\gamma^{-1}(y,z)$. Next we introduce a short exact sequence
$$\begin{tikzcd}
e\arrow{r}{}&\Z\arrow{r}{\iota}&(\Z\times\Z\times\ker(\alpha))\rtimes_\eta\Z\arrow{r}{\varphi}&(\Z\times\ker(\alpha))\rtimes_\gamma\Z\arrow{r}{}&e
\end{tikzcd}.$$
Here $\iota$ is defined by $x\mapsto(x,0,e,0)$ and $\varphi$ is defined by $(x,y,z,w)\mapsto(y,z,w)$. The fact that $\iota$ is a homomorphism follows immediately from the definition of multiplication in the semi-direct product. To show that $\varphi$ is a homomorphism note that 
\begin{align*}
\varphi((x_1,y_1,z_1,w_1)\cdot(x_2,y_2,z_2,w_2))&=\varphi((x_1,y_1,z_1)\cdot\eta^{w_1}(x_2,y_2,z_2),w_1+w_2)\\
&=((y_1,z_1)\cdot\gamma^{w_1}(y_2,z_2),w_1+w_2)\\
&=(y_1,z_1,w_1)\cdot(y_2,z_2,w_2)\\
&=\varphi(x_1,y_1,z_1,w_1)\cdot\varphi(x_2,y_2,z_2,w_2).
\end{align*}
Clearly this is an exact sequence from the definition of the maps. To show that $\iota(1)$ is central it suffices to compute that $\eta(1,0,e)=(1,0,e)$.

To show that the cohomology class of this extension is skinny we make the diagram
$$\begin{tikzcd}
e\arrow{r} & \Z \arrow{r}{} 
\arrow{d}{\id_\Z} &\Z\times\Z\times\ker(\alpha)\arrow{r}{}
\arrow{d}{} 
&\Z\times\ker(\alpha) \arrow{r}{}\arrow{d}{}& e
\\
e\arrow{r}{}&\Z\arrow{r}{\iota}&(\Z\times\Z\times\ker(\alpha))\rtimes_\eta\Z\arrow{r}{\varphi}&(\Z\times\ker(\alpha))\rtimes_\gamma\Z\arrow{r}{}&e.
\end{tikzcd}$$
Note that the image of $\Z\times\ker(\alpha)$ in $(\Z\times\ker(\alpha))\rtimes_\gamma\Z$ is precisely the kernel of the map $(y,z,w)\mapsto w$ from $(\Z\times\ker(\alpha))\rtimes_\gamma\Z$. Thus this extension is skinny with respect to this map by Proposition~\ref{skinnys}. Combining this with the isomorphism to $\Gamma_\sigma$ this map corresponds to $\alpha\circ\varphi_\sigma$.

Define a set-theoretic section $\theta$ so that $(y,z,w)\mapsto(0,y,z,w)$. Then a representative of $\omega\in[\omega]$ may be defined by the formula $\iota(\omega(x,y))=\theta(x)\theta(y)\theta(xy)^{-1}$. Note that we may have picked the isomorphism to the semi-direct product so that $a$ corresponds to $(0,e,1)$. Then we have that
\begin{align*}
\iota(\langle\omega, c_k\rangle)&=\theta(a)\theta(\iota_\sigma(k))\theta(a\iota_\sigma(k))^{-1}\big(\theta(\iota_\sigma(k))\theta(a)\theta(\iota_\sigma(k)a)^{-1}\big)^{-1}\\
&=\theta(a)\theta(\iota_\sigma(k))\theta(a)^{-1}\theta(\iota_\sigma(k))^{-1}\\
&=(0,0,e,1)\cdot(0,k,e,0)\cdot(0,0,e,1)^{-1}\cdot(0,k,e,0)^{-1}\\
&=(\eta(0,k,e),1)\cdot(0,0,e,-1)\cdot(0,-k,e,0)\\
&=(k,\gamma(k,e),1)\cdot(0,0,e,-1)\cdot(0,-k,e,0)\\
&=(k,\gamma(k,e),0)\cdot(0,-k,e,0).
\end{align*}
Because $\iota_\sigma(k)$ corresponds to $(k,e,0)\in(\Z\times\ker(\alpha))\rtimes_\gamma\Z$, and $\iota_\sigma(k)$ is central in $\Gamma_\sigma$ we must have that $\gamma(k,e)=(k,e)$. Thus we have
\begin{align*}
\iota(\langle\omega, c_k\rangle)&=(k,k,e,0)\cdot(0,-k,e,0)\\
&=\iota(k).
\end{align*}
The fact that $[\omega]$ is non-torsion follows from the fact that this pairing is nonzero.
\end{proof}

\begin{lemma}\label{kmap}
Let $\Gamma$ be a discrete group, let $[\sigma]\in H^2(\Gamma)$ and let $k\in\Z$. There is a commutative diagram as follows
$$\begin{tikzcd}
e\arrow{r} & \Z \arrow{r}{\iota_\sigma} 
\arrow{d}{\bullet k} &\Gamma_\sigma\arrow{r}{\varphi_\sigma}
\arrow{d}{\psi} 
&\Gamma\arrow{r}{}\arrow{d}{\id_\Gamma}& e
\\
e\arrow{r} & \Z\arrow{r}{\iota_{k\sigma}}
&
\Gamma_{k\sigma}\arrow{r}{\varphi_{k\sigma}}
&\Gamma \arrow{r}{}& e.
\end{tikzcd}$$
\begin{proof}
The map $\bullet k$ from $\Z$ to $\Z$ induces a map from $H^2(\Gamma)\rightarrow H^2(\Gamma)$. Checking the definition of the bifunctor it is clear that this map must be multiplication by $k$ so it takes $[\sigma]$ to $k[\sigma]$. By the central extension view of bifunctoriality \cite[Chapter IV.3 Exercise 1]{coho} it must take the class corresponding to the unique extension of $\Gamma$ fitting into the bottom row of the commutative diagram in the statement of the Lemma. 
\end{proof}
\end{lemma}
\begin{thm}\label{main1}
Let $\Gamma$ be a torsion-free, finitely generated nilpotent group with Hirsch length at least 2 and let $\alpha\in\Hom(\Gamma,\Z)$ be the canonical homomorphism with respect to some Mal'cev basis of $\Gamma$. Then there is a cohomology class $[\sigma]\in H^2(\Gamma)$ that is skinny with respect to $\alpha$ and non-torsion.
\end{thm}
\begin{proof}
We will proceed by induction on the Hirsch length. For the base case the only torsion-free nilpotent group with Hirsch length 2 (up to isomorphism) is~$\Z^2$. The two elements of the Mal'cev basis must span $\Z^2$ we may assume without loss of generality that our Mal'cev basis is just the standard generators. Then the cohomology class of $\sigma(x,y)=x_2y_1$ is clearly skinny with respect to the canonical homomorphism $\alpha(x)=x_1$. Moreover, this is not torsion because it pairs non-trivially with the homology class of
$$[(0,1)|(1,0)]-[(1,0)|(0,1)].$$

For the inductive step we will assume the theorem for $\Gamma$ and show it for~$\Gamma_\omega$ for all $[\omega]\in H^2(\Gamma)$. This is valid since any torsion-free finitely generated nilpotent group $\Lambda$, with Mal'cev basis $a_1,\ldots, a_m$, is the middle group in some central extension of $\Lambda/\langle a_m\rangle$ and the canonical homomorphism for $a_1,\ldots, a_m$ is the composition of the quotient map with the canonical homomorphism of a Mal'cev basis of $\Lambda/\langle a_m\rangle$, by Proposition~\ref{quotientB}.

Let $[\sigma]\in H^2(\Gamma)$ be a non-torsion cohomology class that is skinny with respect to the canonical homomorphism $\alpha$.

\medskip
\noindent
{\bf Case 1}: $\varphi_\omega^*([\sigma])$ is non-torsion. In this case the result follows from Lemma~\ref{pullback}.

\medskip
\noindent
{\bf Case 2:} $\varphi_\omega^*([\sigma])$ is torsion. In this case we use the fact that $\ker(\varphi_\omega^*)$ is spanned by $[\omega]$, by Proposition~\ref{phiker}. Thus we conclude that $k[\omega]=j[\sigma]$ where $j$ and $k$ are nonzero integers. Let $\psi:\Gamma_\omega\rightarrow\Gamma_{k\omega}$ be the map from Lemma~\ref{kmap}. Note that $j[\sigma]$ is also a skinny cohomology class. Since $\Gamma_{k\omega}\cong\Gamma_{j\sigma}$ we may let $[\gamma]\in H^2(\Gamma_{k\omega})$ be the cohomology class constructed in Lemma~\ref{hard}. Let $b\in\Gamma_\omega$ be an element so that $\alpha(\varphi_\omega(b))=1$. Then let $a=\psi(b)$. Note that $\varphi_{k\omega}(a)=\varphi_\omega(b)$ by the commutative diagram in Lemma~\ref{kmap}. Using the diagram again we get that $\psi(\iota_\omega(1))=\iota_{k\omega}(k)$. From Lemma~\ref{hard} we have that
\begin{align*}
k&=\langle\gamma,[a|\iota_{k\omega}(k)]-[\iota_{k\omega}(k)|a]\rangle\\
&=\langle\gamma,[\psi(b)|\psi(\iota_{\omega}(1))]-[\psi(\iota_{\omega}(1))|\psi(b)]\rangle\\
&=\langle\gamma,\psi_*([b|\iota_\omega(1)]-[\iota_\omega(1)|b])\rangle\\
&=\langle\psi^*(\gamma),[b|\iota_\omega(1)]-[\iota_\omega(1)|b]\rangle.
\end{align*}
Since the pairing on the last line is nonzero we have that $\psi^*([\gamma])$ cannot be torsion. Moreover by Lemma~\ref{pullback} we have that $\psi^*([\gamma])$ is skinny with respect to the homomorphism $\alpha\circ\varphi_{k\omega}\circ\psi=\alpha\circ\varphi_\omega$. 
\end{proof}

\section{Construction of Asymptotic Representations}

\begin{prop}\label{formula}
Let $\Gamma$ be a torsion-free finitely, generated nilpotent group with a specified Mal'cev basis, and let $[\sigma]$ be a cohomology class that is skinny with respect to the canonical homomorphism. By Proposition~\ref{skinnys} there is representative $\sigma$ of $[\sigma]$ so that $\sigma(x,y)=q(x_1,\ldots,x_m,y_1)$ for $q\in\Q[x_1,\ldots,x_m,y_1]$. Then for $n$ co-prime to the denominators of coefficients of $q$ there is a well-defined formula $\rho_n:\Gamma\rightarrow U(\ell^2(\Z/n\Z))$ defined by
$$\rho_n(x)\delta_{\bar j}=\exp\left(\frac{2\pi i}{n}q(x_1,\ldots,x_m,j)\right)\delta_{\bar{j}+\bar{x}_1}.$$
Here $j\in\Z$, $\bar{j}$ represents its reduction mod $n$ and $\delta_{\bar{j}}\in\ell^2(\Z/n\Z)$ is the standard basis element corresponding to $\bar{j}$.
\end{prop}
\begin{proof}This is a modification of the proof of \cite[Proposition 3.17]{constructive} that is pointed out in \cite[Remark 3.18]{constructive}.

We will show that $\rho_n$ is well-defined. This is because if $\bar{j}=\bar{k}$ then $q(x_1,\ldots,x_m,j)\equiv q(x_1,\ldots,x_m,k)\mod n$ so because the exponential depends only on the reduction of $q$ mod $n$ the formula is well-defined. This is a unitary because it takes an orthonormal basis to another orthonormal basis.
\end{proof}

This does not define $\rho_n$ for all $n$, but it does define $\rho_n$ for infinitely many~$n$. This is sufficient to prove that a group is not $p$-stable, because we may pick a sequence of integers $n_j$ enumerating the $n$ for which~$\rho_n$ is well-defined.

\begin{lemma}\label{chiform}
Define
$$\chi_n(x,y)=\exp\left(\frac{2\pi i}{n}\sigma(x,y)\right)$$
The formula $\rho_n$ defined in Proposition~\ref{formula} obeys the equation
$$\rho_n(xy)\rho_n(y)^{-1}\rho_n(x)^{-1}=\chi_n(x,y)^{-1}\id_{\ell^2(\Z/n\Z)}.$$
\end{lemma}
\begin{proof} This is a modification of the proof of \cite[Lemma 3.19]{constructive}.

Note that since $\partial\sigma=0$ we have that for all $x,y,z\in\Gamma$
$$-\sigma(x,yz)+\sigma(xy,z)-\sigma(y,z)=-\sigma(x,y).$$
Exponentiating both sides we get
$$\chi_n(x,yz)^{-1}\chi_n(xy,z)\chi_n(y,z)^{-1}=\chi_n(x,y)^{-1}.$$
Next we claim that
$$\rho_n(x)^{-1}\delta_{\bar y_1}=\chi_n(x,x^{-1}y)^{-1}\delta_{\bar y_1-\bar x_1}.$$
To check this it suffices to compute that
\begin{align*}
\rho_n(x){\chi}_n(x,x^{-1}y)^{-1}\delta_{\bar y_1-\bar x_1}&={\chi}_n(x,x^{-1}y)^{-1}\rho_n(x)\delta_{\bar y_1-\bar x_1}\\
&={\chi}_n(x,x^{-1}y)^{-1}{\chi}_n(x,x^{-1}y)\delta_{\bar y_1}\\
&=\delta_{\bar y_1}.
\end{align*}
Combining these we get
\begin{align*}
\rho_n(xy)\rho_n(y)^{-1}\rho_n(x)^{-1}\delta_{\bar{z}_1}&=\rho_n(xy)\rho_n(y)^{-1}\chi_n(x,x^{-1}z)^{-1}\delta_{\bar{z}_1-\bar{x}_1}\\
&=\rho_n(xy)\chi_n(y,y^{-1}x^{-1}z)^{-1}\chi_n(x,x^{-1}z)^{-1}\delta_{\bar{z}_1-\bar{x}_1-\bar{y}_1}\\
&=\chi_n(xy,y^{-1}x^{-1}z)\chi_n(y,y^{-1}x^{-1}z)^{-1}\chi_n(x,x^{-1}z)^{-1}\delta_{\bar{z}_1}\\
&=\chi_n(x,y)^{-1}\delta_{\bar{z}_1}.
\end{align*}
\end{proof}

\begin{prop}\label{asymp}
Let $\Gamma$, $\sigma$ and $\rho_n$ be defined as in Proposition~\ref{formula}; let $1\le p<\infty$. Then the following bounds hold:
\begin{align*}
||\rho_n(xy)-\rho_n(x)\rho_n(y)||&\le\frac{2\pi|\sigma(x,y)|}{n}\\
||\rho_n(xy)-\rho_n(x)\rho_n(y)||_p&\le2\pi|\sigma(x,y)|n^{\frac1p-1}.
\end{align*}
\end{prop}
\begin{proof}
We compute
\begin{align*}
||\rho_n(xy)-\rho_n(x)\rho_n(y)||&=||\rho_n(xy)\rho_n(y)^{-1}\rho_n(x)^{-1}-\id_{\ell^2(\Z/n\Z)}||\\
&=||(\chi_n(x,y)^{-1}-1)\id_{\ell^2(\Z/n\Z)}||\\
&=|\chi_n(x,y)^{-1}-1|\\
&=\left|\exp\left(\frac{-2\pi i\sigma(x,y)}{n}\right)-1\right|\\
&\le\frac{2\pi|\sigma(x,y)|}{n}.
\end{align*}
For the Schatten $p$-norm note that
\begin{align*}
||\rho_n(xy)-\rho_n(x)\rho_n(y)||_p&\le||\rho_n(xy)-\rho_n(x)\rho_n(y)||n^{1/p}\\
&\le 2\pi|\sigma(x,y)|n^{\frac1p-1}.
\end{align*}
\end{proof}

\begin{thm}\label{main2}
If $\Gamma$ is a finitely generated torsion-free nilpotent group with a non-torsion skinny cohomology class, and $p<1\le\infty$ then $\Gamma$ is not $p$-stable. The sequence of maps $\rho_n$ defined in Proposition~\ref{formula} is a $p$-asymptotic representation that is not $p$-perturbable to a genuine representation.
\end{thm}
\begin{proof}This is a modification of the proof of \cite[Theorem 3.20]{constructive}.

Asymptotic multiplicativity follows from Proposition~\ref{asymp}.

Now we will show that $\rho_n$ cannot be perturbed to a genuine representation in the operator norm, which implies that it cannot be perturbed in any unnormalized Schatten $p$-norm because all the operator norm bounds all other Schatten $p$-norms below. To that end we claim that there is some element $c\in H_2(\Gamma)$ so that $\langle\sigma,c\rangle\ne0$. To see this we note that from the universal coefficient theorem \cite[Theorem~53.1]{eat}  we have a short exact sequence
$$\begin{tikzcd}
0\arrow{r}{}&\Ext(H_1(\Gamma),\Z)\arrow{r}{}& H^2(\Gamma)\arrow{r}{}&\Hom(H_2(\Gamma),\Z)\arrow{r}{}&0
\end{tikzcd}.$$
Note that since $\Gamma$ is finitely generated $H_1(\Gamma)\cong\Gamma/[\Gamma,\Gamma]$ \cite[page 36]{coho} is finitely generated as well. Thus $\Ext(H_1(\Gamma),\Z))$ can be show to be torsion from \cite[Theorem 52.3]{eat} and the table on \cite{eat} page~331. Thus since $\sigma$ is non-torsion it must pair nontrivially with some element $c$ of the 2-homology which we write
$$c=\sum_{j=1}^Nx_j[a_j|b_j].$$

Note that $\rho_n$ is asymptotically multiplicative in operator norm. Thus the pairing $\langle\rho_n,c\rangle$ defined in
Definition~\ref{pairing} will be defined for large enough $n$. For such $n$ we compute
\begin{align*}
\langle\rho_n,c\rangle&=\frac{1}{2\pi i}\sum_{j=1}^N x_j\Tr(\log(\rho_n(a_jb_j)\rho_n(b_j)^{-1}\rho_n(a_j)^{-1}))\\
&=\frac{1}{2\pi i}\sum_{j=1}^N x_j\Tr(\log(\chi_n(a_jb_j)^{-1}\id_{\ell^2(\Z/n\Z)}))&\mbox{by Lemma~\ref{chiform}}\\
&=-\frac{1}{2\pi i}\sum_{j=1}^N x_j\frac{2\pi i}{n}\sigma(a_j,b_j)\Tr(\id_{\ell^2(\Z/n\Z)})\\
&=-\langle\sigma,c\rangle\\
&\ne0.
\end{align*}
Then by Theorem~\ref{obstruction} $\rho_n$ cannot be within $\frac{1}{24}$ of a genuine representation in the operator norm on the set $\{a_j,b_j,a_jb_j\}_{j=1}^N$. It follows that it cannot be within $\frac{1}{24}$ of a genuine representation in the Schatten $p$-norm on this set either.
\end{proof}

Next we will prove Theorem~\ref{main3}:

\begin{proof}
Let $\Gamma$ be a finitely generated nilpotent group of that is not locally cyclic. By Proposition~\ref{vcyc}~$\Gamma$ has Hirsch length at least 2. Let $T$ be the subset of torsion elements. This is a finite normal subgroup by Proposition~\ref{torsion}. By Definition~\ref{hirsch2}, $\Gamma/T$ has a Hirsch length of at least 2 as well. By Theorem~\ref{main1} and Theorem~\ref{main2} we know that $\Gamma/T$ is not $p$-stable. By Lemma~\ref{quotient} we have that $\Gamma$ cannot be $p$-stable either.

Now let $\Gamma$ be a virtually cyclic group. Thus $\Gamma$ is a virtually free group so we may apply Lemma~\ref{vfree} to show that $\Gamma$ is Frobenius stable.
\end{proof}

\begin{rmk}
If $\Gamma$ is a finitely generated nilpotent group that is not virtually cyclic a $p$-asymptotic representation that is not $p$-perturbable to a genuine representation can be found as follows. If $T$ is the torsion subgroup of~$\Gamma$, apply Proposition~\ref{formula} to $\Gamma/T$ to define $\rho_n:\Gamma/T\rightarrow U(n)$. Then if $f:\Gamma\rightarrow\Gamma/T$ is the quotient map $\rho_n\circ f$ is the desired sequence of maps.
\end{rmk}

\begin{center}
\textbf{Acknowledgments}
\end{center}
I would like to thank my advisor Marius Dadarlat for pointing me towards this problem and for his advice along the way. I would also like to thank an anonymous referee taking the time to read the paper closely and providing many helpful comments and suggestions. A question posed by the referee and a conversation with Marius Dadarlat and Iason Moutzouris helped me to locate the right result in the literature to drop the torsion-free assumption from the main result; I would like to thank all three. Finally, I would like to thank the Purdue Mathematics Department for supporting me with the summer research grant for the summer of 2022.
\medskip
\bibliographystyle{plain}
\bibliography{main}
\end{document}